\documentclass[12pt]{article}
\usepackage{amsmath,amsthm}
\usepackage{amsfonts}
\usepackage{latexsym}
\usepackage{indentfirst}
\usepackage{psfrag,epsf}
\usepackage{pstool}

\usepackage{graphicx, epsfig, subfigure}
\usepackage{epstopdf}
\usepackage{enumerate}
\usepackage[subnum]{cases}
\usepackage{geometry}
\usepackage{fullpage}

\newtheorem{theorem}{Theorem} \rm
\newtheorem{lemma}[theorem]{Lemma}

\newtheorem{conjecture}[theorem]{Conjecture}

\newtheorem{claim}[theorem]{Claim}

\newtheorem{problem}[theorem]{Problem}

\newtheorem{remark}[theorem]{Remark}
\newtheorem{question}[theorem]{Question}
\theoremstyle{plain}

\title{The square of every subcubic planar graph of girth at least 6 is 7-choosable}
\author{Seog-Jin Kim\thanks{Department of Mathematics Education, Konkuk University,
Korea.  E-mail: {\tt skim12@konkuk.ac.kr}}, 
\and~Xiaopan Lian\thanks{Center for Combinatorics and LPMC, Nankai University, China.
{E-mail: }{\tt lianxiaopanlxp@163.com}}
}

\begin{document}

\maketitle

\begin{abstract}
The square of a graph $G$, denoted $G^2$, has the
same vertex set as $G$ and has an edge between two vertices if the distance between
them in $G$ is at most $2$.  Thomassen \cite{Thomassen}  and Hartke, Jahanbekam and Thomas \cite{Hartke} 
proved that $\chi(G^2) \leq 7$ if $G$ is a subcubic planar graph.  A natural question is whether $\chi_{\ell}(G^2) \leq 7$ or not if $G$ is a subcubic planar graph.  It was showed in \cite{CK} that $\chi_{\ell}(G^2) \leq 7$ if $G$ is a subcubic planar graph of girth  at least 7.
We prove that  $\chi_{\ell}(G^2) \leq 7$ if
$G$ is a subcubic planar graph of girth at least 6.
\end{abstract}

\section{Introduction}

The {\em square} of a graph $G$, denoted $G^2$, has the
same vertex set as $G$ and has an edge between two vertices if the distance between
them in $G$ is at most $2$.
We say a graph $G$ is  {\em subcubic} if $\Delta(G) \leq 3$, where $\Delta(G)$ is the maximum degree in $G$.  The  {\em girth} of $G$, denoted $g(G)$, is the size of smallest cycle in $G$.
Let $\chi(G)$ be the  chromatic number of a graph $G$.

Wegner \cite{Wegner} posed the following conjecture.

\begin{conjecture}\cite{Wegner} \label{conj-Wegner}
Let $G$ be a planar graph. The chromatic number
$\chi(G^2)$ of $G^2$ is at most 7 if $\Delta(G) = 3$,
at most $\Delta(G)+5$ if $4 \leq \Delta(G) \leq 7$, and at most $\lfloor \frac{3 \Delta(G)}{2} \rfloor$ if $\Delta(G) \geq 8$.
\end{conjecture}

Conjecture \ref{conj-Wegner} is still wide open.  The only case for which we know tight bound is when $\Delta(G) = 3$. Thomassen \cite{Thomassen} showed that $\chi(G^2) \leq 7$ if $G$ is a planar graph with $\Delta(G)  = 3$, which implies that Conjecture \ref{conj-Wegner} is true for $\Delta(G)  = 3$. Conjecture \ref{conj-Wegner} for $\Delta(G)  = 3$ is also confirmed by 
Hartke, Jahanbekam and Thomas \cite{Hartke}. The proof in \cite {Thomassen} relies on a detailed structural analysis, and the proof in \cite{Hartke} uses discharging argument with extensive
computer case-checking.  

Many results were obtained with conditions on $\Delta(G)$.
Bousquet, Deschamps, Meyer and Pierron \cite{BDMP} showed that $\chi(G^2) \leq 12$ if $G$ is a planar graph with $\Delta(G) \leq 4$, and 
Hou, Jin, Miao,  and Zhao \cite{HJMZ} showed that $\chi(G^2) \leq 18$ if $G$ is a planar graph with $\Delta(G) \leq 5$.  Also, Bousquet, Deschamps, Meyer and Pierron \cite{BDMP21} showed that $\chi(G^2) \leq 2 \Delta(G) + 7$ if $G$ is a planar graph with
$6 \leq \Delta(G) \leq 31$. 
For general $\Delta(G)$, best known upper bound is that $\chi(G^2) \leq \lceil \frac{5 \Delta(G)}{3}\rceil + 78$ by Molloy and Salavatipour \cite{MS05}.  On the other hand, Havet, van den Heuvel, McDiarmid, and Reed \cite{Havet} proved that
Conjecture \ref{conj-Wegner} holds asymptotically.
One may see detail story on the study of Wegner's conjecture in \cite{Cranston22}.

A list assignment for a graph is a function $L$ that assigns each vertex a list of
available colors. The graph is $L$-colorable if it has a proper coloring $f$ such that
$f (v) \in L(v)$ for all $v$. A graph is called $k$-choosable if $G$ is $L$-colorable whenever
all lists have size $k$. The list chromatic number $\chi_{\ell}(G)$ is the minimum $k$ such that $G$
is $k$-choosable.

Since it was known in \cite{Thomassen} that $\chi(G^2) \leq 7$ if $G$ is a subcubic planar graph, the following natural question was raised in \cite{CK}.
\begin{question} \cite{CK} \label{CK-question}
Is it true that  $\chi_{\ell} (G^2) \leq 7$ if $G$ is a subcubic planar graph?
\end{question}

For general upper bound on $\chi_{\ell}(G^2)$ for a subcubic graph $G$,
Cranston and Kim \cite{CK} proved that $\chi_{\ell} (G^2) \leq 8$ if $G$ is a connected graph (not necessarily planar) with $\Delta(G) = 3$ and if $G$ is not the Petersen graph.  And Cranston and Kim \cite{CK} proved that $\chi_{\ell} (G^2) \leq 7$ if $G$ is a  subcubic planar graph with $g(G) \geq 7$.

For a subcubic planar graph $G$ and for $k \in \{4, 5, 6, 7, 8\}$,  the best girth conditions known to imply that $\chi_{\ell}(G^2) \leq k$ are as follows (\cite{Cranston22}).
\begin{equation} \label{table-girth}
\begin{array}{ccccccc}
\chi_{\ell}(G^2) \leq & | & 8 & 7 & 6  & 5 & 4 \\
 \hline
g(G) \geq & | & 3 & 7 & 9 & 13 & 24
\end{array}
\end{equation}
Cranston and Kim \cite{CK} and Havet \cite{Havet09} showed that $\chi_{\ell}(G^2) \leq 6$ if $G$ is a subcubic planar graph with $g(G) \geq 9$.  Havet \cite{Havet09} showed that  $\chi_{\ell}(G^2) \leq 5$ if $G$ is a subcubic planar graph with $g(G) \geq 13$.  And
Borodin, Ivanova, and Neustroeva \cite{Borodin} and Havet \cite{Havet09} showed that
$\chi_{\ell}(G^2) \leq 4$ if $G$ is a subcubic planar graph with $g(G) \geq 24$.

\medskip

In this paper, we consider subcubic planar graphs.  To deduce that $\chi_{\ell}(G^2) \leq 7$, we 
improve the hypothesis $g(G) \geq 7$, shown in table (\ref{table-girth}), to $g(G) \geq 6$.
We prove the following main theorem. 
\begin{theorem} \label{main-thm}
If $G$ is a subcubic planar graph with girth at least 6, then $\chi_{\ell}(G^2) \leq 7$.
\end{theorem}

Note that motivated by the List Total Coloring Conjecture, the following interesting conjecture was proposed in \cite{KW}. \\

\noindent {\bf List Square Coloring Conjecture} \cite{KW} For every graph $G$, we have $\chi_{\ell}(G^2) = \chi(G^2)$. \\

Thus Question \ref{CK-question} was naturally asked in \cite{CK}.  
However, the List Square Coloring Conjecture was disproved in \cite{KP}. Note that a positive result for the List Square Coloring Conjecture for a special class of graphs is still interesting.
It was conjectured in \cite{Havet} that $\chi_{\ell}(G^2) = \chi(G^2)$ if $G$ is a planar graph.  But, recently Hasanvand \cite{MH} proved that there exists a cubic claw-free planar
graph $G$ such that $\chi(G^2) = 4 < \chi_{\ell}(G^2)$.

\begin{remark} \rm
In \cite{CK}, Cranston and Kim showed that if $G$ is a minimal couterexample to Theorem \ref{main-thm}, then
the maximum average degree ($mad(G))$ is at least $\frac{14}{5}$, and then showed that the girth of $G$ is at most 6, which leads a contradiction.  But, applying the maximum average degree of $G$ to prove Theorem \ref{main-thm} is not helpful since $G$ is subcubic.

We use recoloring method to prove the main lemma (Lemma \ref{main-lemma}).
The procedure of proof of Lemma \ref{main-lemma} is as follows.  If $G$ is a minimal counterexample to Theorem \ref{main-thm} and $G$ has a 6-cycle $C$ containing a 2-vertex $u$, then obtain a proper subgraph $H = G - u$ by removing the 2-vertex $u$.  Next, we consider a coloring on $H^2$, and uncolor the vertices on the 6-cycle $C$, and then recolor the vertices on $C$ to have a proper coloring of $G^2$, which is a contradiction.
\end{remark}

\section{Proof of Theorem \ref{main-thm}}
In this section, let $G$ be a minimal counterexample to Theorem \ref{main-thm}.
It means that for any proper subgraph $H$ of $G$, $\chi_{\ell}(H^2) \leq 7$, but $\chi_{\ell}(G^2) > 7$.  A vertex of degree $k$ is called a  $k$-vertex.
First, we prove the following main lemma.

\begin{lemma} \label{main-lemma}
$G$ has no 6-cycle which contains a 2-vertex.
\end{lemma}
\begin{proof}
Suppose that $G$ has a 6-cycle $C$ which contains a 2-vertex.
We denote $V(C) = \{v_1, v_2, v_3, v_4, v_5, v_6\}$ where $v_6$ is the 2-vertex.
(see Figure \ref{6-cycle}.)

Let $L$ be a list assignment with lists of size 7 for each vertex in $G$.
Let $H = G - v_6$.  Then since $G$ is a minimal counterexample to Theorem \ref{main-thm}, the square of $H$ has a proper coloring $\phi$ such that $\phi(v) \in L(v)$ for each vertex $v \in V(H)$.

If $\phi(v_1) \neq \phi(v_5)$, then since $v_6$ has only 6 neighbors in $G^2$, we can complete a proper coloring for  $G^2$ by coloring $v_6$ by a color in $L(v_6)$.  This is a contradiction since $G$ is a counterexample to Theorem \ref{main-thm}.

Therefore, we can assume that $\phi(v_1) = \phi(v_5) = \alpha$ for some color $\alpha$.
And we assume that
\begin{equation} \label{color-assign}
\phi(v_2)  = a, \ \phi(v_3)  = b, \ \phi(v_4)  = c.
\end{equation}

Note that $|\{a,b,c,\alpha\}|=4$ since $\phi$ is a proper coloring of the square of $H$.  Now uncolor the vertices in $V(C)\setminus\{v_6\} = \{v_1, v_2, v_3, v_4, v_5\}$, and
investigate the color lists which are available at each vertex $v$ in $V(C)$.

For each $v \in V(C)$, let $C(v)$ be the color list which is available at $v$.  If we denote by $|C(v)|$ the number of available colors after uncoloring the vertices of $V(C)$, then we have the following information.\[
|C(v_1)| \ge 3, \ |C(v_2)| \ge 2, \ |C(v_3)| \ge2, |C(v_4)| \ge 2, |C(v_5)| \ge 3, |C(v_6)| \ge 5.
\]

\begin{figure}[htbp]
\begin{center}
    \includegraphics[scale=0.32]{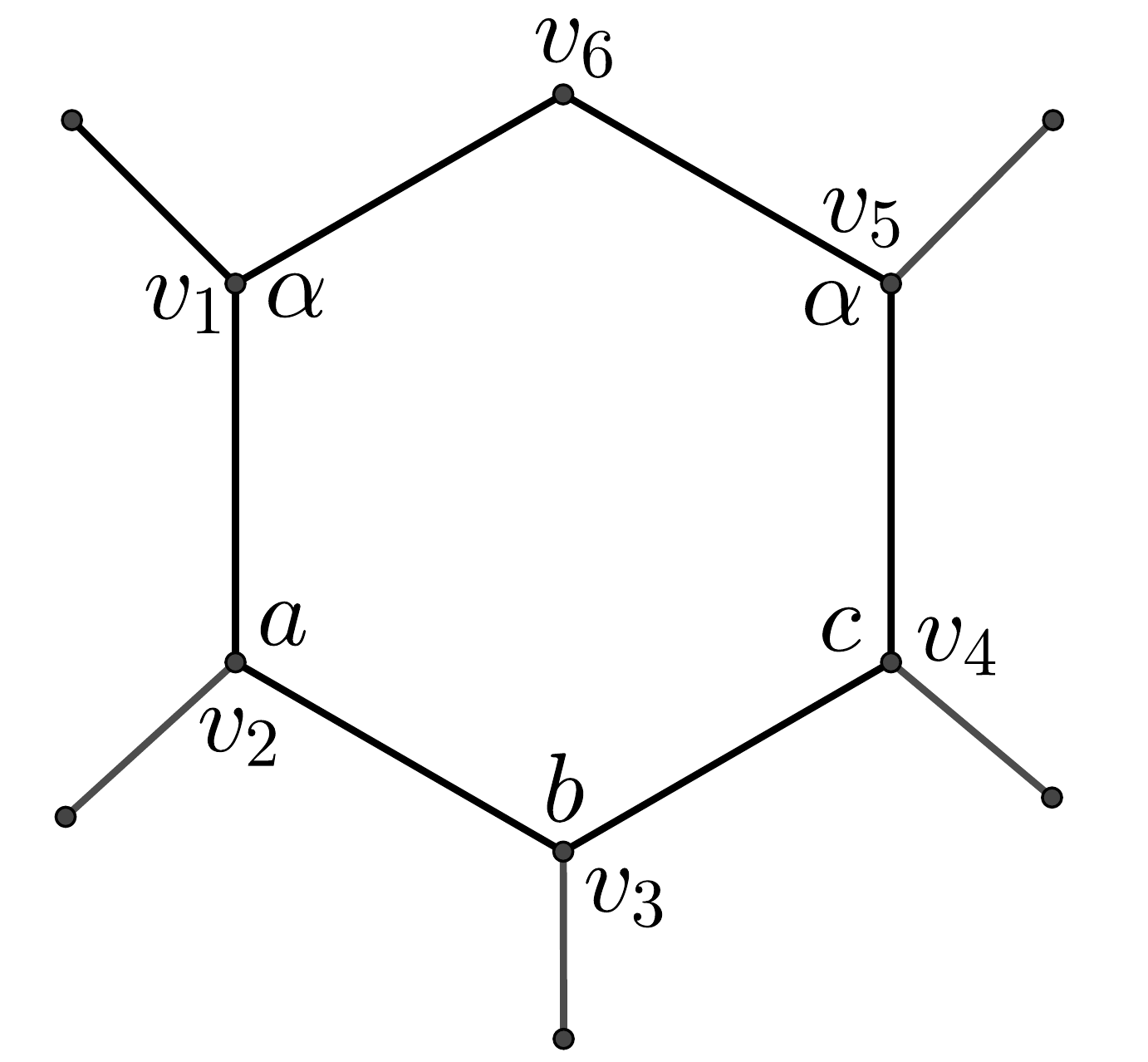}
\caption{6-cycle $C$ and coloring of the square of $H = G - v_6$} \label{6-cycle}
\end{center} 
\end{figure} 

First, to investigate $C(v_1)$ and $C(v_5)$, we claim the following holds.

\begin{claim} \label{claim-one}
By the coloring $\phi$ on the square of $H$, we have that
\[
\begin{aligned}
L(v_1)&=\{\phi(y):y\in V(G)\cap V(H)~\text{and}~v_1y\in E(G^2)\},\\
L(v_5)&=\{\phi(y):y\in V(G)\cap V(H)~\text{and}~v_5y\in E(G^2)\}.
\end{aligned}\]
\end{claim}
Suppose to the contrary that $L(v_1)\neq\{\phi(y):y\in V(G)\cap V(H)~\text{and}~v_1y\in E(G^2)\}$. Let $\gamma\in L(v_1)\setminus\{\phi(y):y\in V(G)\cap V(H)~\text{and}~v_1y\in E(G^2)\}$.
 Then we recolor vertex $v_1$ by $\gamma$ and greedily color vertex $v_6$.  This contradicts the fact that $\chi_\ell(G)>7$. Therefore, $L(v_1)=\{\phi(y):y\in V(G)\cap V(H)~\text{and}~v_1y\in E(G^2)\}$. By the same argument, we can show that $L(v_5)=\{\phi(y):y\in V(G)\cap V(H)~\text{and}~v_5y\in E(G^2)\}$.  Hence Claim \ref{claim-one} holds. \qed

\medskip
Therefore, by the definition of $C(v)$ for $v\in V(C)$ and Claim~\ref{claim-one}, we have that
\begin{equation} \label{list-v1}
C(v_1) = \{a, b, \alpha\}, \  C(v_5)  = \{b, c, \alpha\} \mbox{ for some color } \alpha.
\end{equation}

\medskip
 Next, we investigate $C(v_2)$, $C(v_3)$ and $C(v_4)$. Observe that, by the definition of $C(v)$ for $v\in V(C)$, $\phi(v)\in C(v)$. Then we prove the following claim.

\begin{claim} \label{list-v2}
$C(v_2)\subseteq\{a,b,c\}$, $C(v_3)\subseteq\{a,b,c,\alpha\}$, and $C(v_4)\subseteq\{a,b,c\}$.
\end{claim}
Suppose to the contrary that $C(v_2) \setminus \{a,b, c\}\neq\emptyset $. Then there is a color $\gamma \in C(v_2) \setminus \{a, b, c\}$. So, we recolor $v_2$ by color $\gamma$ and recolor $v_1$ by color $a\in C(v_1)$. Then we greedily color $v_6$. Thus, we have a proper coloring of $G^2$.  This is a contradiction.
Therefore, $C(v_2) \subseteq \{a,b,c\}$.  By the same argument, we can show that $C(v_3)\subseteq\{a,b,c,\alpha\}$, and $C(v_4)\subseteq\{a,b,c\}$.  
Hence Claim \ref{list-v2} holds. \qed

\medskip

Since $\phi(v)\in C(v)$ and $|C(v)|\ge 2$ for $v\in \{v_2,v_3,v_4\}$, by Claim~\ref{list-v2}, it suffices to consider that
the possible lists for $v_2, v_3, v_4$ are as follows.
\begin{equation} \label{v234}
\begin{array}{ccc}
C(v_2) &  C(v_3) & C(v_4) \\
\{a, b\} & \{b, a\} & \{c, a\} \\
\{a, c\} & \{b, c\} & \{c, b\} \\
 & \{b, \alpha\} &
\end{array}
\end{equation}

\medskip

Now, we recolor the vertices in $V(C) =  \{v_1, v_2, v_3, v_4, v_5, v_6\}$ from $C(v)$ for $v \in V(C)$ and obtain a proper coloring of $G^2$, which leads to a contradiction.  From now on, we denote the new coloring by $f$.  Hence, \[
\begin{array}{ll}
(a) \ f(v) = \phi(v),&  \mbox{ if } v \in V(G) \setminus V(C), \\
(b) \ f(v) \in C(v), &  \mbox{ if } v \in V(C).
\end{array}
\]

\bigskip
Now we will show that we have a proper coloring of $G^2$ from the lists in (\ref{list-v1}) and the table (\ref{v234}).  Since $|C(v_6)|\ge 5$, we can first color each vertex in $V(C)\setminus \{v_6\}=\{v_1,v_2,v_3,v_4,v_5\}$ and then we choose a color from $C(v_6)\setminus \{f(v):v\in \{v_1,v_2,v_4,v_5\}\}$ for vertex $v_6$.  Note that $v_1 v_4, \ v_2 v_5 \notin E(G^2)$
since $g(G) \geq 6$.  So, we can assign a same color at $v_1$ and $v_4$ ($v_2$ and $v_5$, respectively) in the new coloring $f$. 

\medskip

\noindent {\bf Case 1:} $C(v_2) = \{a, b\}$

Subcase 1.1.  When $C(v_4) = \{c, a\}$\\
We have the following recoloring on $\{v_1, v_2, v_3, v_4, v_5\}$ which produces a proper coloring of $G^2$.
\begin{equation*}
\begin{array}{ccccl}
C(v_2) &  C(v_3) & C(v_4) \\
 & \{b, a\} &  & \rightarrow  &  f(v_1)=\alpha, f(v_2) = b, f(v_3) = a, f(v_4) = c,  f(v_5) = b\\
\{a, b\} & \{b, c\} & \{c, a\} & \rightarrow  &  f(v_1) = a, f(v_2) = b, f(v_3) = c, f(v_4) = a,  f(v_5) = \alpha \\
 & \{b, \alpha\} &  & \rightarrow  & f(v_1)  = a, f(v_2) = b, f(v_3) = \alpha, f(v_4) = c, f(v_5) = b\\
\end{array}
\end{equation*}

Subcase 1.2. When $C(v_4) = \{c, b\}$\\
We have the following recoloring on $\{v_1, v_2, v_3, v_4, v_5\}$ which produces a proper coloring of $G^2$.
\begin{equation*}
\begin{array}{ccccl}
C(v_2) &  C(v_3) & C(v_4) \\
 & \{b, a\} &  & \rightarrow  & f(v_1)=\alpha,  f(v_2) = b, f(v_3) = a, f(v_4) = c,  f(v_5) = b\\
\{a, b\} & \{b, c\} & \{c, b\} & \rightarrow  &  f(v_1) = b, f(v_2) = a, f(v_3) = c, f(v_4) = b,  f(v_5) = \alpha\\
 & \{b, \alpha\} &  & \rightarrow  & f(v_1)  = a, f(v_2) = b, f(v_3) = \alpha, f(v_4) = c, f(v_5) = b\\
\end{array}
\end{equation*}

\noindent {\bf Case 2:} $C(v_2) = \{a, c\}$

Subcase 2.1. When $C(v_4) = \{c, a\}$\\
We have the following recoloring on $\{v_1, v_2, v_3, v_4, v_5\}$ which produces a proper coloring of $G^2$.
\begin{equation*}
\begin{array}{ccccl}
C(v_2) &  C(v_3) & C(v_4) \\
 & \{b, a\} &  &   &  \\
\{a, c\} & \{b, c\} & \{c, a\} & \rightarrow  &   f(v_1) = a, f(v_2) = c, f(v_3) = b, f(v_4) = a,  f(v_5) = \alpha \\
 & \{b, \alpha\} &  &   &
\end{array}
\end{equation*}

Subcase 2.2. When $C(v_4) = \{c, b\}$\\
We have the following recoloring on $\{v_1, v_2, v_3, v_4, v_5\}$ which produces a proper coloring of $G^2$.
\begin{equation*}
\begin{array}{ccccl}
C(v_2) &  C(v_3) & C(v_4) \\
 & \{b, a\} &  & \rightarrow  &f(v_1) = \alpha,  f(v_2) = c, f(v_3) = a, f(v_4) = b,  f(v_5) = c\\
\{a, c\} & \{b, c\} & \{c, b\} & \rightarrow  &   f(v_1) = b, f(v_2) = a, f(v_3) = c, f(v_4) = b,  f(v_5) = \alpha\\
 & \{b, \alpha\} &  & \rightarrow  & f(v_1)  = a, f(v_2) = c, f(v_3) = \alpha, f(v_4) = b,   f(v_5) = c\\
\end{array}
\end{equation*}
Thus, from Case 1 and Case 2, we have a proper coloring for $G^2$, which is a contradiction. Hence, $G$ has no 6-cycle which contains a 2-vertex.  This completes the proof of Lemma \ref{main-lemma}.
\end{proof}

We can see easily that the following lemma holds.

\begin{lemma} \label{minimum-degree}
$G$ has no 1-vertex.
\end{lemma}

The following lemma was proved in \cite{CK}.  We include it here for readers.

\begin{lemma}(Lemma 13 \cite{CK}) \label{7-reducible}
Let $G$ be a {\em minimal} graph such that $\chi_{\ell}(G^2) > 7$.  
For each vertex $v$, let $M_1(v)$ and $M_2(v)$ be the number of 2-vertices at distance 1 and distance 2 from $v$.
If $v$ is a 3-vertex, then $2M_1(v)+M_2(v)\leq 2$.  If $v$ is a 2-vertex, then $2M_1(v)+M_2(v)=0$.  
\end{lemma}

From Lemma \ref{7-reducible}, we have the following lemma for 2-vertices. 

\begin{lemma} \label{distance-2-vertex}
For every cycle $C$ in $G$, the distance between any two 2-vertices on $C$ is at least 4.
\end{lemma}
\begin{proof}
Let $x$ and $y$ be two 2-vertices.  Then $2M_1(x)+M_2(x)=0$ by Lemma \ref{7-reducible}.
So, the distance $x$ and $y$ must be at least 3.  If the distance between $x$ and $y$ is 3, then 
we have an $x,y$-path, $x w_1 w_2 y$, where $w_1$ and $w_2$ are 3-vertices.  But, in this case, $2M_1(w_1)+M_2(w_1) \geq 3 > 2$, which violates Lemma \ref{7-reducible}.
Thus the distance between $x$ and $y$ is at least 4.
\end{proof}

Before we prove Theorem \ref{main-thm}, we prove the following lemma.

\begin{lemma} \label{2-vertex-no-cutvertex}
If $u$ is a 2-vertex in $G$, then $u$ is not a cut-vertex in $G$.
\end{lemma}
\begin{proof}
Suppose that $u$ is a cut-vertex in $G$.  Then $G - u$ is not a connected graph. 
Let $N_G(u) = \{x, y\}$ and let $H = G - u + xy$.  That is, $H$ is the resulting graph obtained from $G$ by removing $u$ and making $x$ and $y$ adjacent.  Note that $H$ is still subcubic and the girth of $H$ is at least 6.  Then since $|V(H)| = |V(G)| - 1$ and $G$ is a minimal counterexample to Theroem \ref{main-thm}, $H^2$ has a proper coloring $\phi$.  Note that $\phi(x) \neq \phi(y)$.  
Since $u$ has at most 6 neighbors in $G^2$, we have a proper coloring of $G^2$ by coloring greedily $u$.  This is a contradiction.  Hence a 2-vertex $u$ is not a cut-vertex in $G$.
\end{proof}

Now we prove the main theorem.

\medskip
\noindent {\bf Theorem 3.}
If $G$ is a subcubic planar graph with girth at least 6, then $\chi_{\ell}(G^2) \leq 7$.
\begin{proof}
Let $G$ be a minimal counterexample to the theorem and let $G$ be a plane graph drawn on the plane without crossing edges.  Let $F(G)$ be the set of faces of $G$. For a face $C \in F(G)$, let $d(C)$ be the length of $C$.  

We assign $2d(x)-6$ to each vertex $x \in V(G)$ and $d(x) - 6$ for each face $x \in F(G)$ as an original charge function $\omega(x)$ of $x$.
According to Euler's formula $|V(G)| - |E(G)| + |F(G)| = 2$,
we have
\begin{equation} \label{eqn1}
\sum_{x\in V(G)\cup F(G)}\omega(x) = \sum_{v\in V(G)}(2d(v)-6)+\sum_{f\in F(G)}(d(f)-6) = -12. 
\end{equation}
We next design some discharging rules to redistribute charges along the graph with conservation of the total charge. Let $\omega'(x)$ be the charge of $x \in V(G)\cup F(G)$ after the discharge procedure such that ${\displaystyle \sum_{x\in V(G)\cup F(G)}\omega(x)=\sum_{x\in V(G)\cup F(G)}\omega'(x)}$. Next, we will show that $\omega'(x)\ge 0$ for all $x\in V(G)\cup F(G)$, which leads the following contradiction.
\[
0\le \sum_{x\in V(G)\cup F(G)}\omega'(x)=\sum_{x\in V(G)\cup F(G)}\omega(x) = \sum_{v\in V(G)}(2d(v)-6)+\sum_{f\in F(G)}(d(f)-6) = -12.
\]

\medskip
Observe that $G$ has no 1-vertex by Lemma~\ref{minimum-degree}.
Thus, we have the following discharging rule.

\medskip
\noindent {\bf The discharging rule:}
\begin{enumerate}[(R1)]
\item If a 2-vertex $u$ is on a face $C$, then $C$ gives charge 1 to $u$.
\end{enumerate}
\medskip

Now, we will show that the new charge $\omega'(x) \geq 0$ for every $x \in V(G) \cup F(G)$.

\medskip
\noindent
(1) When $x \in V(G)$ \\
If $d(x)  = 2$, then  $\omega(x)  = -2$.
By (R1),  $x$ receives charge 1 from each of its incident faces.  So, $\omega'(x) = 0$.
Note that every 2-vertex $x$ is incident to two faces by Lemma \ref{2-vertex-no-cutvertex}.
If $d(x) = 3$, then $\omega(x) = \omega'(x) = 0$.

\medskip
\noindent (2) When $x \in F(G)$ \\
Here we denote $x$ by a face $C$.
If $C$ is a 6-cycle, then $C$ has no 2-vertex by Lemma \ref{main-lemma}.  So
 $\omega(C) = \omega'(C)  = 0$.

 \medskip 
To complete the case (2), we first prove the following claim.

\begin{claim} \label{number-2-vertex}
For every face $C$ of $G$, there are at most at most $\lfloor\frac{d(C)}{4}\rfloor$ 2-vertices on the boundary of $C$.
\end{claim} 
\begin{proof}
Let $W$ be the closed walk which is the boundary of $C$. Let $D_1,\ldots,D_k$ be the cycles of $G$ contained in $W$. Thus, by Lemma~\ref{distance-2-vertex}, for each $i\in \{1,\ldots,k\}$, $D_i$ contains at most $\lfloor \frac{d(D_i)}{4} \rfloor$ 2-vertices. Since $G$ is a subcubic graph, by Lemma~\ref{2-vertex-no-cutvertex}, the 2-vertices are only possibly contained in the cycles $D_1,\ldots,D_k$. Hence, the number of 2-vertices contained in $C$ is at most $ \sum^k_{i=1}\lfloor \frac{|E(D_i)|}{4} \rfloor\le \lfloor\frac{\sum^k_{i=1}{|E(D_i)|}}{4}\rfloor\le \lfloor\frac{|E(W)|}{4}\rfloor= \lfloor\frac{d(C)}{4} \rfloor$.
Therefore, there are at most $\lfloor\frac{d(C)}{4}\rfloor$ 2-vertices on the boundary of $C$.
Thus Claim \ref{number-2-vertex} holds.
\end{proof}
 
Now, we consider the case when $d(C)\ge 7$. If $C$ is a 7-cycle, then $C$ has at most one 2-vertex. So, $\omega'(C) \geq d(C) - 6 - 1 = 0$.
If $C$ is a cycle of length at least 8, then $\omega'(x) \geq 0$ since $d(C) - 6 - \lfloor \frac{d(C)}{4} \rfloor \geq 0$. Hence  $\omega'(x) \geq 0$ for every $x \in F(G)$.

\medskip
Hence, by (1) and (2), we have that $\omega'(x) \geq 0$ for every $x \in V(G) \cup F(G)$.  This completes the proof of Theorem \ref{main-thm}. 
\end{proof}

\section{Future work}
We proved that $\chi_{\ell}(G^2) \leq 7$ if $G$ is a subcubic planar graph of girth at least 6.
But, we do not know yet whether Question \ref{CK-question} is true or not.
Hence answering Question \ref{CK-question} is interesting.  Or as a weaker version, we can ask the following problem.

\begin{problem}
Is it true that $\chi_{\ell}(G^2) \leq 7$ if $G$ is a subcubic planar graph of girth at least 4?
\end{problem}

\section*{Acknowledgments}
We thank  Daniel Cranston for his very helpful comments and suggestions.
The first author is supported by the National Research Foundation of Korea(NRF) grant funded by the Korea government(MSIT)(No.NRF-2021R1A2C1005785),
and the second author is partially supported by the National Natural Science Foundation of China (No. 12161141006).



\end{document}